\documentclass[11pt,a4paper,reqno]{amsart}

\usepackage[T1]{fontenc}
\usepackage{microtype, a4wide, textcomp,fbb}
\usepackage{
    amsmath,  amssymb,  amsthm,   amscd,
    gensymb,  graphicx, etoolbox, 
    booktabs, stackrel, mathtools    
}
\usepackage[usenames,dvipsnames]{xcolor}
\definecolor{darkblue}{rgb}{0,0,0.7}
\definecolor{darkred}{rgb}{0.7,0,0}
\usepackage[colorlinks=true, linkcolor=darkred, citecolor=darkblue, urlcolor=blue, pagebackref=true, breaklinks=true]{hyperref}
\usepackage[capitalise]{cleveref}
\usepackage{placeins}
\usepackage{relsize}
\usepackage{todonotes}
\usepackage{soul}

\usepackage{tikz}
\usetikzlibrary{arrows}
\usetikzlibrary{shapes}
\usepackage{float}
\usepackage{tabularx}
\usepackage{kantlipsum}
\usepackage{array}
\usepackage[margin=2.6cm]{geometry}

\newtheorem{proposition}{Proposition}[section]
\newtheorem{lemma}[proposition]{Lemma}
\newtheorem{theorem}[proposition]{Theorem}

\newtheorem{corollary}[proposition]{Corollary}
\newtheorem{question}[proposition]{Question}

\theoremstyle{definition}
\newtheorem{remark}[proposition]{Remark}

\newtheorem{example}[proposition]{Example}
\newtheorem{definition}[proposition]{Definition}

\newtheorem{innercustomthm}{Theorem}
\newenvironment{customthm}[1]
  {\renewcommand\theinnercustomthm{#1}\innercustomthm\itshape}
  {\endinnercustomthm}

\newcommand{\reg}{{\rm reg}}

\tikzstyle{place}=[draw,circle,minimum size=1mm,inner sep=1pt,outer sep=-1.1pt,fill=black]

\usetikzlibrary{shapes.geometric}
\tikzstyle{places}=[draw,rectangle,minimum size=8pt,inner sep=0pt]
\tikzstyle{placesf}=[draw,rectangle,minimum size=5pt,inner sep=0pt]
\tikzstyle{placec}=[draw,circle,minimum size=8pt,inner sep=0pt]
\tikzstyle{placecf}=[draw,circle, minimum size=5pt,inner sep=0pt]

\def\K{\mathbb{K}}
\def\G{\mathcal{G}}
\def\D{\mathcal{D}}
\def\reg{\mathrm{reg}}
\def\pd{\mathrm{pd}}
\def\l{\langle}
\def\r{\rangle}

\begin{document}

\title{Castelnuovo-Mumford regularity of the closed neighborhood ideal of a graph}
%\thanks{Last updated: \today}

\author{Shiny Chakraborty}
\address{The University of Texas at Dallas}
\email{shiny.chakraborty@utdallas.edu}

\author{Ajay P. Joseph}
\address{National Institute of Technology Karnataka, India}
\email{ajaymath.217ma001@nitk.edu.in}

\author{Amit Roy}
\address{Chennai Mathematical Institute, India}
\email{amitiisermohali493@gmail.com}

\author{Anurag Singh}
\address{Indian Institute of Technology Bhilai, India}
\email{anurags@iitbhilai.ac.in}


\vspace*{-0.4cm}
\begin{abstract}
Let $G$ be a finite simple graph and let $NI(G)$ denote the closed neighborhood ideal of $G$ in a polynomial ring $R$. We show that if $G$ is a forest, then the Castelnuovo-Mumford regularity of $R/NI(G)$ is the same as the matching number of $G$, thus proving a conjecture of Sharifan and Moradi in the affirmative. We also show that the matching number of $G$ provides a lower bound for the Castelnuovo-Mumford regularity of $R/NI(G)$ for any $G$. Furthermore, we prove that, if $G$ contains a simplicial vertex, then $NI(G)$ admits a Betti splitting, and consequently, we show that the projective dimension of $R/NI(G)$ is also bounded below by the matching number of $G$, if $G$ is a forest or a unicyclic graph. 

\end{abstract}

\maketitle

\section{Introduction} 
Combinatorial commutative algebra is an exciting branch of research living at the junction of three areas in mathematics - algebra, combinatorics, and topology. Bridging these areas is a class of homogeneous ideals namely, the square-free monomial ideals. For a long time now, commutative algebraists have been interested in studying the properties of finite, simple graphs using the monomial ideals as a tool, aspiring to make a dictionary between the world of combinatorics and that of algebra. It is along this direction that Villarreal's work \cite{CM} showed how the edges of a finite, simple graph are used to construct a class of monomial ideals, known as the edge ideal, and how the properties of this associated ideal can be studied using the properties of the graph and vice versa. Since then, different kinds of monomial ideals of polynomial rings associated with graphs have been carefully studied. Some of them are cover ideals \cite{FS,AVT}, path ideals \cite{AB,HVT}, t-clique ideals \cite{MKA,SM2}, splitting monomial ideals \cite{BoD,FHVT}, weighted oriented edge ideals \cite{BCDMS, MPo}, Stanley-Reisner ideals of $r$-independence complexes of graphs \cite{DRSV23, PaRo}, etc. 

 The study of such ideals has instilled curiosity and enthusiasm to delve further into the techniques and methodologies in extracting information about the structure of the underlying quotient modules. Sharifan and Moradi \cite{SM} introduced a special class of square-free monomial ideal, viz. the closed neighborhood ideal $NI\left(G\right)$ of a graph $G$. Let $G$ be a finite simple graph with the vertex set $V(G)=\{x_1,\ldots,x_n\}$ and let $R_G=\mathbb K[x_1,\ldots,x_n]$ be the polynomial ring in $n$ variables over a field $\mathbb K$. We simply use $R$ in place of $R_G$ whenever the underlying graph $G$ is clear. The ideal $NI(G)$ is formed by associating a monomial to each closed neighbor of the vertices of G (see \Cref{definition}). The Castelnuovo-Mumford regularity, simply called the regularity, and the projective dimension of $R/NI(G)$ are two important numerical invariants of the $R$-module $R/NI(G)$. In general, if $M$ is a finitely generated $R$-module, then the projective dimension and the regularity of $M$ help us to understand the structure of the minimal free resolution of $M$, which in turn, gives us useful information about the underlying module. Sharifan and Moradi analyzed the projective dimension $\pd(R/NI(G))$ and the regularity $\reg(R/NI(G))$ of $R/NI(G)$, by relating them to combinatorial properties of the graph $G$, particularly, its matching number $a_G$. Following this, a considerable amount of research has been done to understand the structure and properties of the closed neighborhood ideals from algebraic and combinatorial viewpoints.

In \cite{HW}, Honeycutt and Sather-Wagstaff described the minimal irreducible decompositions of $NI(G)$ in terms of the minimal dominating sets of $G$. They also characterized the trees whose closed neighborhood ideals are Cohen–Macaulay. Along this direction, Leamann \cite{JL} characterized all chordal and bipartite graphs that have Cohen-Macaulay closed neighborhood ideals. Nasernejad, Qureshi, Bandari, and Musapasaouglu in \cite{NQBM} investigated the normally torsion-free property and the (strong) persistence property of the closed neighborhood ideals of trees and cycles. Moreover, Nasernejad and Qureshi proved in \cite{NQ} that the closed neighborhood ideals of strongly chordal graphs are normally torsion-free. Nasernejad, Bandari, and Roberts found a sufficient criterion for the normality of arbitrary monomial ideals in \cite{NBR}, and as an application, they proved that the closed neighborhood ideals of complete bipartite graphs are normal. Recently, the structure of the minimal generating sets of the Alexander dual of $NI(G)$ was considered in \cite{NMe}.

Our interest in this direction has been mainly driven by the work of Sharifan and Moradi in \cite{SM}. They investigated the Castelnuovo-Mumford regularity and the projective dimension of the closed neighborhood ideals of various classes of graphs; for example, the generalized star graphs, forests, $m$-book graphs, etc. For a forest $G$, it was proved in \cite[Theorem 2.5]{SM} that both the regularity and projective dimension of $R/NI(G)$ are bounded below by the matching number of $G$. 
The authors further conjectured \cite[Conjecture 2.11]{SM} that for forests, the regularity of $R/NI(G)$ is the same as the matching number of $G$. In this paper, we prove this conjecture. More specifically, we prove the following.

\begin{theorem}\label{conjecture proof}
    Let $G$ be a forest and let $\reg(R/NI(G))$ denote the Castelnuovo-Mumford regularity of $R/NI(G)$. Then, $\reg(R/NI(G))= a_G$, where $a_G$ denotes the matching number of $G$.
\end{theorem}

We also show that the regularity part of result \cite[Theorem 2.5]{SM} can be generalized to all graphs using a recent topological result by Matsushita and Wakatsuki \cite{MaWa}. In particular, we prove the following.

\begin{theorem}\label{thm:conj}
    For any graph $G$, $\reg(R/NI(G))\ge a_G$, where $a_G$ is the matching number of $G$.
\end{theorem}

Motivated by this, we examine the relationship between  $\pd(R/NI(G))$, and the matching number $a_G$ for various families of $G$. As our first result in this direction, we show that if $G$ is a graph containing a simplicial vertex $x$, then $NI(G)$ admits a Betti splitting (cf. \Cref{chordal closed neighborhood}). Furthermore, if $x$ is a leaf vertex of $G$ and $\pd(R_{G'}/NI(G'))\ge a_{G'}$, then we have $\pd(R_G/NI(G))\ge a_G$, where $G'=G\setminus N_G[x]$ . As a corollary of this result (see \Cref{forest pd}), we get an alternate proof of the fact that $\pd(R/NI(G))\ge a_G$, when $G$ is a forest. In \Cref{rem:pdoftree}, we show that, unlike the regularity case, the inequality concerning the projective dimension of forests and the matching number could be strict. In \Cref{unicyclic}, we prove that if $G$ is a unicyclic graph, then $\pd(R/NI(G))\ge a_G$. Moreover, in \Cref{cycle remark}, we show that both the inequalities involving projective dimension, regularity, and matching number of unicyclic graphs in \Cref{unicyclic} and \Cref{thm:conj} could be strict. In \Cref{chordal not bounded below}, we point out that, unlike the regularity, the projective dimension of the closed neighborhood ideal of a graph is not necessarily bounded below by the matching number of $G$ even if $G$ is chordal. \Cref{wheel result} contains a formula for $\pd(R_G/NI(G))$ in terms of the matching number $a_G$, where $G$ is a wheel graph. In \Cref{whisker}, we show that if $G$ is any graph and $G'$ is a graph obtained from $G$ by attaching a whisker on each vertex of $G$, then $\reg(R/NI(G'))=\pd(R/NI(G'))=a_{G'}$. 

This article is organized as follows. In \Cref{section 2}, we recall relevant definitions and auxiliary results from graph theory and commutative algebra. We also state and prove some preliminary results concerning the closed neighborhood ideals that will be used in the subsequent section. In \Cref{section 3}, we give proofs of all our results stated above. Finally, in \Cref{section 4}, we outline some questions for future research.

\section{Preliminaries}\label{section 2}
In this section, we recall preliminary notions and terminologies from graph theory and commutative algebra that will be used throughout the article. In subsection \ref{graph}, we give various definitions related to graphs. In subsection \ref{algebra}, we set forth the ground of basic homological algebra terminology necessary to understand the main results of the paper. In subsection \ref{known results}, we present a few definitions pertaining to closed neighborhood ideals and state some results that will be used in \Cref{section 3}. 

\subsection{Graph theoretic notions}\label{graph}

A {\it graph} $G$ is a pair $(V (G), E(G))$, where $V(G)$ is the set of vertices and $E(G) \subseteq V(G)\times V(G)$ is the set of edges. For simplicity of notations, we sometimes write $V$ and $E$ to denote $V(G)$ and $E(G)$, respectively. A graph is {\it finite } if it has finite number of vertices and edges.
%An edge connnecting a vertex to itself is called a loop and multiple edges are two or more edges connecting the same two vertices.
A graph without loops and multiple edges is called a {\it simple graph}.
%We will be considering only finite, simple graphs throughout this paper.
The {\it complement of $G$}, denoted by $G^c$, is a graph whose vertex set is $V(G)$ and edge set is $\{\{x,y\}: \{x,y\}\notin E(G)\}$. If $x$ is a vertex of $G$, then by $G\setminus x$ we mean the graph on the vertex set $V\setminus\{x\}$ with edge set $\{\{a,b\}\in E: x\notin \{a,b\}\}$. Similarly, if $x_1,\ldots,x_r$ are some vertices of $G$, then $G\setminus \{x_1,\ldots,x_r\}$ denotes the graph on the vertex set $V\setminus \{x_1,\ldots,x_r\}$ with edge set $\{\{a,b\}\in E: x_i\notin \{a,b\}\text{ for each }i\in[r]\}$. For a vertex $u$ of $G$, the set $\{v \in V : \{u,v\}\in E\}$ is called the set of {\it neighbors} of $u$, and is denoted by $N_G(u)$. 
 The set of \textit{closed neighbors} of $u$ in $G$ is $N_G(u)\cup \{u\}$ and is denoted by $N_G[u]$. Whenever the underlying graph $G$ is clear, we denote the set $N_G(u)$ and $N_G[u]$ by $N(u)$ and $N[u]$, respectively. The number $|N(u)|$ is called the degree of $u$ and is denoted by $\deg(u)$. If for some $u\in V(G)$, $\deg(u)=1$, then $u$ is called a {\it leaf} of $G$. Given a graph $G$, to add a {\it whisker} at a vertex $x$ of $G$, one simply adds a new vertex $y$ and an edge connecting $y$ and $x$ to $G$. Note that in the new graph, $y$ is a leaf.

A {\it cycle} $C_k$ of length $k$ is a graph on the vertex set $\{x_1,\ldots,x_k\}$ with edge set $\{\{x_1,x_k\},\{x_i,x_{i+1}\}: 1\le i\le k-1\}$. Let $G$ be any arbitrary graph, and let $A$ be a subset of $V(G)$. Then, the {\it induced subgraph of $G$ on $A$}, denoted
$G[A]$, is the graph with vertex set $V(G[A]) = A$ and edge
set $E(G[A]) = \{e \in E(G) : e \subseteq A\}$. 
Let $A=\{x_{i_1},\ldots,x_{i_k}\}\subseteq V(G)$ be such that $G[A]\cong C_k$ and $E(G[A])=\{\{x_{i_k},x_{i_1}\},\{x_{i_j},x_{i_{j+1}}\}: 1\le j\le k-1\}$, then we simply write the cycle $G[A]$ as $x_{i_1}\cdots x_{i_k}$. A graph is called {\it unicyclic} if it is connected and contains exactly one cycle. A \textit{forest} is a graph without any cycle. A connected forest is called a {\it tree}. A graph on the vertex set $\{x_1,\ldots,x_n\}$ is called a {\it path graph} of length $n-1$, denoted $P_n$, if $E(P_n)=\{\{x_i,x_{i+1}\}: 1\le i\le n-1\}$. The graph $P_n$ is an example of a tree. The star graph of order $n$, denoted $S_{n}$, is a tree with $n$ vertices among which one vertex has its degree $n-1$, while the other $n-1$ vertices each have degree $1$. 

A graph $G$ is called \textit{chordal} if it contains no induced cycle of length $4$ or more. Note that forests are, in particular, chordal graphs.
Another important class of chordal graphs is the complete graphs. For a positive integer $n$, a \textit{complete graph} $K_{n}$ is a graph on $n$ vertices such that there is an edge between any two distinct vertices. If $G$ is a chordal graph, then each induced subgraph $H$ of $G$ contains at least one vertex $z$ such that $N_H[z]$ is a complete graph (see \cite{Dirac}). Such a vertex is called a {\it simplicial vertex}. 

A graph $G$ is said to be a \textit{bipartite graph} if its vertex set $V$ can be partitioned into two disjoint subsets $V_{1}$ and $V_{2}$ such that each edge of $G$ connects a vertex of $V_{1}$ to a vertex of $V_{2}$. If every vertex of $V_1$ is connected with every vertex of $V_2$, then we say that the bipartite graph is {\it complete} and it is denoted by $K_{m,n}$, where $|V_1|=m$ and $|V_2|=n$. Here, we remark that cycles of even lengths are examples of bipartite graphs.

A subset $M\subseteq E(G)$ is called a \textit{matching} of $G$ if for all distinct $e,e'\in M$, $e\cap e'=\emptyset$. A matching of $G$ is called {\it maximal} if it is not properly contained in any other matching of $G$. The maximum size of a matching in $G$ is called the \textit{matching number} of $G$, and it is denoted by $a_G$. 

The \textit{chromatic number} of a graph $G$, denoted $\chi(G)$, is the minimum number of colors required to color the vertices of $G$ in such a way that adjacent vertices receive different colors. 
%The notion of matching number plays an important role in this paper. More specifically, it provides a lower bound on the Castelnuovo-Mumford regularity and projective dimension of the closed neighborhood ideals for various families of graphs.

\subsection{Algebraic preliminaries}\label{algebra}
    
Let $\mathbb K$ be a field and let $R=\mathbb K[x_1,\ldots,x_n]$ be the polynomial ring in $n$ variables endowed with the usual $\mathbb N$-grading. That is, if $m=\prod_{i=1}^nx_i^{\alpha_i}$ is a monomial of $R$, where $\alpha_i\in\mathbb N \cup \{ 0 \}$, then $\deg(m)=\sum_{i=1}^n\alpha_i$. The monomial $m$ is called a {\it square-free monomial} if $\alpha_i\in\{0,1\}$ for all $1\le i\le n$. An ideal $I$ is said to be a {\it monomial ideal} if $I$ is generated by monomials. We say that $I$ is a {\it square-free monomial ideal} if $I$ is generated by square-free monomials. For a monomial ideal $I$, the unique set of minimal generators of $I$ is denoted by $\G(I)$ \cite{RHV}.

Now let $I$ be a monomial ideal of $R$. Then, $R/I$ is a graded $R$-module. The {\it graded minimal free resolution} of $R/I$ is the long exact sequence 
\[
\mathcal F_{\cdot}: \,\, 0\rightarrow F_p\xrightarrow{\partial_{p}}\cdots\xrightarrow{\partial_{2}} F_1\xrightarrow{\partial_1} F_0\xrightarrow{\eta} R/I\rightarrow 0, 
\]
where $F_0=R$, $F_i=\bigoplus_{j \in \mathbb{N}}
R(-j)^{\beta_{i,j}(R/I)}$ for $i\ge 1$, and $\eta$ is the quotient map. Here
 $R(-j)$ denotes the polynomial ring $\mathbb K[x_1,\ldots,x_n]$ with the grading
twisted by $j$, and $p \leq n$ (by Hilbert's syzygy theorem).  The numbers $\beta_{i,j}(R/I)$ are called the {\it $(i,j)^{\text{th}}$ graded Betti numbers of $R/I$}. They play a crucial role in characterizing the structure of the module $R/I$. See \cite{RHV} for more on graded free resolutions.

 The {\it projective dimension} of $R/I$, denoted by $\pd(R/I)$, is the number $\max\{i: \beta_{i,j}(R/I)\neq 0 \text{ for some } j \}$. The invariant $\pd(R/I)$ plays an important role in analyzing the structure of the minimal free resolution of $R/I$. Another important numerical invariant of the graded module $R/I$ is its {\it Castelnuovo-Mumford regularity}, denoted by $\reg(R/I)$, and is defined as follows:
    \[
    \reg(R/I) = \max\{ j-i : \beta_{i,j}(R/I) \neq 0\}.
    \]

    \subsection{Closed neighborhood ideal}\label{known results}
    We now formally introduce the closed neighborhood ideal of a graph as defined by Sharifan and Moradi in \cite{SM}.  Let $G=(V(G),E(G))$ be a finite simple graph with $V(G)=\{x_1,\ldots,x_n\}$. For $w\in V(G)$, define the monomial $m_{(G,w)}$ associated with the closed neighborhood of $w$ as $m_{(G,w)}=\prod_{x_{i}\in N_G[w]}x_{i}$ in the polynomial ring $\K[x_1,\ldots,x_n]$. Given a graph $G$ on the vertex set $\{x_1,\ldots, x_n\}$, we sometimes denote the polynomial ring $\mathbb K[x_1,\ldots, x_n]$ as $R_G$. Now the closed neighborhood ideal $NI(G)$ of $G$ is a square-free monomial ideal in $R_G$ defined as follows. 

    \begin{definition}\label{definition}
        If $G$ is a finite simple graph, then the closed neighborhood ideal of $G$ is defined as 
   \[
    NI(G)=\langle \{ m_{(G,w)}: w\in V(G) \} \rangle.
    \]
    \end{definition}
    
Our aim in this article is to determine the numbers $\reg(R_G/NI(G))$ and $\pd(R_G/NI(G))$ for various families of $G$, and we do so by relating it to the graph theoretical invariant, matching number $a_G$ of $G$. A quick observation is that not every monomial $m_{(G,w)}$ is in the minimal generating set of $NI(G)$. For example, if $u$ and $v$ are two vertices of $G$ such that $N_G[u]\subsetneq N_G[v]$, then $m_{(G,v)}$ is not among the minimal generators of $NI(G)$. Building on this observation we note that if $G$ is any graph and $G'$ is obtained from $G$ by attaching a whisker at every vertex of $G$, then $\reg(R_{G'}/NI(G'))= a_{G'}$ and $\pd(R_{G'}/NI(G'))= a_{G'}$. To prove this result, we first need the following lemma related to the regularity and projective dimension of the sum of two ideals.

    \begin{lemma}\label{reg sum}\cite[cf. Lemma 2.4]{ZXWT}
        Let $I_1\subseteq R_1=\mathbb K[x_1,\ldots,x_n]$ and $I_2\subseteq R_2=\mathbb K[y_1,\ldots, y_m]$ be two homogeneous ideals. Consider the ideal $I=I_1R+I_2R\subseteq R=\mathbb K[x_1,\ldots,x_n,y_1,\ldots,y_m]$. Then, 
        \[
        \reg(R/I)=\reg(R_1/I_1)+\reg(R_2/I_2),\,\,\text{and}\,\,\pd(R/I)=\pd(R_1/I_1)+\pd(R_2/I_2).
        \]
    \end{lemma}
\begin{proof}
        Let $\mathcal F_{\cdot}$ and $\mathcal G_{\cdot}$ be minimal free resolutions of $R_1/I_1$ and $R_2/I_2$, respectively. Then, the minimal free resolution of $R/I$ is the tensor product of $\mathcal F_{\cdot}$ and $\mathcal G_{\cdot}$. Thus we have, for all $i,j$, 
        \[
        \beta_{i,j}(R/I)=\sum_{\underset{q+q'=j}{p+p'=i}}\beta_{p,q}(R_1/I_1)\beta_{p',q'}(R_2/I_2).
        \]
        Hence, $\reg(R/I)=\reg(R_1/I_1)+\reg(R_2/I_2)$ and $\pd(R/I)=\pd(R_1/I_1)+\pd(R_2/I_2)$.
    \end{proof}
We are now ready to give proof of the above-discussed result.
\begin{proposition}\label{whisker}
    Let $G$ be any simple graph and suppose $G'$ is obtained from $G$ by attaching a whisker at every vertex of $G$. Then, $\reg(R_{G'}/NI(G'))= a_{G'}$ and $\pd(R_{G'}/NI(G'))=a_{G'}$.
\end{proposition}
\begin{proof}
    Suppose $V(G)=\{x_1,\ldots,x_n\}$ and $V(G')=\{x_1,\ldots,x_n,y_1,\ldots,y_n\}$. Thus $\{x_i,y_i\}$ are the whiskers added to $G$. Then, for each $1\le i\le n$, $N_{G'}[y_i]\subseteq N_{G'}[x_i]$. Therefore, the closed neighborhood ideal
\[
    NI(G')=\langle \{ m_{(G',y_i)}: 1\le i\le n \} \rangle.
\]
Here $m_{(G',y_i)}=x_iy_i$. Thus $NI(G')=\langle \{ x_iy_i: 1\le i\le n 
\} \rangle$. If $G_i$ denotes the induced subgraph of $G'$ on the vertex set $\{x_i,y_i\}$, then $NI(G')=\sum_{i=1}^n NI(G_i)$. By \Cref{reg sum}, $\reg(R_{G'}/NI(G'))=\sum_{i=1}^n\reg(R_{G_i}/NI(G_i))$ and $\pd(R_{G'}/NI(G'))=\sum_{i=1}^n\pd(R_{G_i}/NI(G_i))$. Now
\[
    0\rightarrow R_{G_i}(-2)\xrightarrow{\mu_i}R_{G_i}\xrightarrow{\eta_i}R_{G_i}/NI(G_i)\rightarrow 0
    \]
    is a graded minimal free resolution of $R_{G_i}/NI(G_i)$, where $\mu_i$ is the multiplication map by $x_iy_i$ and $\eta_i$ is the quotient map. Therefore, $\reg(R_{G_i}/NI(G_i))=\pd(R_{G_i}/NI(G_i))=1$ for each $1\le i\le n$. Hence, $\reg(R_{G'}/NI(G'))=\pd(R_{G'}/NI(G'))=n$.

    Now our aim is to show that $a_{G'}=n$. Let $M_{G'}=\{\{x_i,y_i\}: 1\le i\le n\}$. Then, $M_{G'}$ is a maximal matching of $G'$. We proceed to show that $|M_{G'}|=a_{G'}$. The proof is by induction on $n$. If $n\le 2$, then it is easy to see that $|M_{G'}|=a_{G'}$. Now suppose $n\ge 3$. Let $M$ be a matching of $G'$ such that $\{x_i,x_j\}\in M$ for some $1\le i<j\le n$. Then, for each $w\in N_{G'}(x_i)$ and $z\in N_{G'}(x_j)$, $\{w,x_i\}\notin M$ and $\{z,x_j\}\notin M$. Define $\widehat G=G\setminus\{x_i,x_j\}$ and construct a new graph $\widetilde G$ by attaching a whisker at every vertex of  $\widehat G$. In other words, $\widetilde G=G'\setminus\{x_i,y_i,x_j,y_j\}$. Now by induction hypothesis, $|M_{\widetilde G}|=a_{\widetilde G}$, where $M_{\widetilde G}=\{\{x_t,y_t\}: 1\le t\le n, t\neq i\text{ and }t\neq j\}$. Since $M\setminus \{\{x_i,x_j\}\}$ is a matching of $\widetilde G$ , $|M_{\widetilde G}|\ge |M|-1$. Observe that $|M_{\widetilde G}|=|M_{G'}|-2$. Hence, $|M_{G'}|\ge |M|$, and consequently,  $a_{G'}=|M_{G'}|=n$. Therefore, $\reg(R_{G'}/NI(G'))= a_{G'}$ and $\pd(R_{G'}/NI(G'))= a_{G'}$.
\end{proof}
We need the following lemma in \Cref{section 3}.

\begin{lemma}\label{extra variable}\cite[cf. Lemma 2.5]{ZXWT}
    Let $I\subseteq R=\mathbb K[x_1,\ldots,x_n]$ be a homogeneous ideal and let $I'=x_{n+1}I\subseteq R'=\mathbb K[x_1,\ldots,x_n,x_{n+1}]$. Then, $\beta_{i,j}(R/I)=\beta_{i,j+1}(R'/I')$ for $i\ge 1$. In particular, $\reg(R'/I')=\reg(R/I)+1$ and $\pd(R'/I')=\pd(R/I)$.
\end{lemma}
%\begin{proof}
%Let 
%\[
%\mathcal F_{\cdot}: \,\, 0\rightarrow F_p\xrightarrow{\partial_{p}}\cdots\xrightarrow{\partial_{2}} F_1\xrightarrow{\partial_1} F_0\xrightarrow{\eta} R/I\rightarrow 0 
%\]
%be the minimal free resolution of $R/I$, where $F_0=R$, $F_i=\bigoplus_{j \in \mathbb{N}}
%R(-j)^{\beta_{i,j}(R/I)}$ for $i\ge 1$, and $\eta$ is the quotient map. Then,
%\[
%\mathcal F_{\cdot}': \,\, 0\rightarrow F_p'\xrightarrow{\partial_{p}'}\cdots\xrightarrow{\partial_{2}'} F_1'\xrightarrow{\partial_1'} F_0'\xrightarrow{\eta'} R'/I'\rightarrow 0 
%\]
%is a minimal free resolution of $R'/I'$, where $F_0'=R'$, $F_i'=\bigoplus_{j\in\mathbb N}R'(-j-1)^{\beta_{i,j}(R/I)}$ for $i\ge 1$, and the differential matrices of $\partial_i'=\partial_i$ for $i\ge 2$ and $\partial_1'=x_{n+1}\partial_1$. Note here that if $I=\langle 
%u_1,\ldots,u_k \rangle$ with $\deg(u_i)=t_i$, then $F_1=\oplus_{i=1}^kR(-t_i)$ and $F_1'=\oplus_{i=1}^kR'(-t_i-1)$. Therefore, $\beta_{i,j}(R/I)=\beta_{i,j+1}(R'/I')$ for $i\ge 1$ and hence, $\reg(R'/I')=\reg(R/I)+1$ and $\pd(R'/I')=\pd(R/I)$.

%\end{proof}

An ideal related to the closed neighborhood ideal of a graph is the {\it edge ideal}. The edge ideal $I(G)$ of a graph $G$ is the ideal generated by the monomials $x_ix_j$ corresponding to all the edges $\{x_i,x_j\}$ of $G$. In 1988 Fr\"oberg \cite{RF} gave an algebraic characterization of the chordal graphs in the context of edge ideals. More precisely, he proved the following.

\begin{theorem}[Fr\"oberg's theorem]
A graph $G$ is chordal if and only if $\reg(R_{G^c}/I(G^c))=1$.
\label{Froberg}
\end{theorem}

Note that if $S_n$ is a star graph on $n$ vertices, then $NI(S_n)=I(S_n)$. Also, it is easy to see that $a_{S_n}=1$. Now since $S_n^c$ is a chordal graph, by Fr\"oberg's theorem, $\reg(R_{S_n}/NI(S_n))=\reg(R_{S_n}/I(S_n))=1$. Therefore, we have,
\begin{corollary}\label{star}
    If $S_n$ denotes the star graph on $n$ vertices, then $\reg(R_{S_n}/NI(S_n))=a_{S_n} = 1$.    
\end{corollary} 

In \Cref{section 3} we show that $\pd(R_{S_n}/NI(S_n))\ge a_{S_n}$. Note that $S_3$ is the path graph on $3$ vertices. We show in the next section that $\pd(R_{S_3}/NI(S_3))> a_{S_3}$.

\section{Regularity and projective dimension of closed neighborhood ideals}\label{section 3}

%In this section, we consider the closed neighborhood ideals of chordal graphs, unicyclic graphs, and complete bipartite graphs. We show that the Castelnuovo-Mumford regularity of these ideals is bounded below by the matching number of the corresponding graphs. Moreover, for one particular class of chordal graphs, namely forests, we show that the regularity of the closed neighborhood ideal is the same as their matching number. For the class of unicyclic graphs and forests, we show that the projective dimension of their closed neighborhood ideals is also bounded below by the matching number of the corresponding graphs. We also consider the closed neighborhood ideal of the corona product and the join of two graphs and provide a lower bound for the regularity of the ideal in terms of the matching number of the corresponding graphs. As an application, we obtain the matching number of a wheel graph as a lower bound for the regularity of the closed neighborhood ideal.

In \cite[Theorem 2.5]{SM} Sharifan and Moradi proved that if $G$ is a forest, then $\reg(R_G/NI(G))\ge a_G$. Moreover, they conjectured \cite[Conjecture 2.11]{SM} that, for a forest $G$, $\reg(R_G/NI(G))=a_G$. As one of the main results of this section, we prove this conjecture.

We first recall the following result which will be used to compute the regularity of the closed neighborhood ideals of various families of graphs.

\begin{lemma}\textup{\cite[cf. Lemma 2.10]{DHS}}
    Let $I\subseteq R =\K[x_1,\ldots,x_n]$ be a square-free monomial ideal and let $x_i$ be a variable appearing in some generator of $I$. Then,
    \begin{enumerate}
        \item[(i)] $\reg(R/\langle I,x_i \rangle\le \reg(R/I)$,
        \item[(ii)] $\reg(R/I)\le\max\{\reg(R/(I:x_i))+1,\reg(R/\langle I,x_i\rangle)\}$.
    \end{enumerate}
    
    \label{lemma1}
    \end{lemma}
    
     Note that the connected components of a forest are trees. The following lemma tells us that in order to find a relationship between $\reg(R_G/NI(G))$ and $a_G$ when $G$ is a forest, it is enough to take $G$ to be a tree.

        \begin{lemma}\label{component lemma}
        Let $G$ be a finite simple graph with connected components $\mathcal{C}_1,\ldots,\mathcal{C}_k$. If $\mathrm{reg}(R_{\mathcal{C}_i}/NI(\mathcal C_i))\ge a_{\mathcal{C}_i}$ for each $i=1,2,\ldots,k$, then $\reg(R_G/NI(G))\ge a_G$. Similarly, if $\mathrm{reg}(R_{\mathcal{C}_i}/NI(\mathcal C_i))\le a_{\mathcal{C}_i}$ for each $i=1,2,\ldots,k$, then $\reg(R_G/NI(G))\le a_G$.
    \end{lemma}

    \begin{proof}
        Note that if $M_i$ is a matching of $\mathcal{C}_i$ such that $a_{\mathcal{C}_i}=|M_i|$ for each $i\in [k]$, then $\sqcup_{i=1}^kM_i$ is a matching of $G$. Hence, $a_{G}\ge \sum_{i=1}^k a_{\mathcal{C}_i}$. Now let $M$ be a matching of $G$ such that $a_G=|M|$. Then, $M\cap E(\mathcal{C}_i)$ is a matching of $\mathcal{C}_i$  for each $i=1,2,\ldots,k$. Hence, $\sum_{i=1}^k a_{\mathcal{C}_i}\ge \sum_{i=1}^k |(M\cap E(\mathcal C_i))|=|M|=a_G$. Thus we have $a_G=\sum_{i=1}^k a_{\mathcal{C}_i}$.

        Now since $\mathcal{C}_i$'s are the connected components of $G$, we have $NI(G)=\sum_{i=1}^kNI(\mathcal{C}_i)$. Therefore, by \Cref{reg sum}, $\mathrm{reg}(R_G/NI(G))=\sum_{i=1}^k\mathrm{reg}(R_{\mathcal{C}_i}/NI(\mathcal C_i))$. Hence, if $\mathrm{reg}(R_{\mathcal{C}_i}/NI(\mathcal C_i))\ge a_{\mathcal{C}_i}$ for each $i\in[k]$, then $\reg(R_G/NI(G))=\sum_{i=1}^k\mathrm{reg}(R_{\mathcal{C}_i}/NI(\mathcal C_i))\ge \sum_{i=1}^k a_{\mathcal{C}_i} = a_G$. Similarly, if $\mathrm{reg}(R_{\mathcal{C}_i}/NI(\mathcal C_i))\le a_{\mathcal{C}_i}$ for each $i\in[k]$, then $\reg(R_G/NI(G))=\sum_{i=1}^k\mathrm{reg}(R_{\mathcal{C}_i}/NI(\mathcal C_i))\le \sum_{i=1}^k a_{\mathcal{C}_i}= a_G$.
    \end{proof}

In the next theorem, we show that for trees, the Castelnuovo-Mumford regularity of the closed neighborhood ideal is the same as its matching number.

 Recall that a tree $T$ is a finite simple graph such that there exists a unique path between
any two distinct vertices. Fix a vertex $z$ of $T$, and we say that $T$ is a rooted tree with $z$ being the root of $T$. Now for any $y\in V(T)$ if $z=w_{i_0},w_{i_1},\ldots,w_{i_{r-1}},w_{i_r}=y$ is the unique path between $z$ and $y$, then we say that $y$ has {\it level} $r$ and denote it by $\mathrm{level}(y)=r$. The level of the root vertex is defined to be zero. We define the height of $T$ to be 
\[
\mathrm{height}(T)=\max_{y\in V(T)}\mathrm{level}(y).
\]
Note that each tree can be realized as a rooted tree by fixing a vertex of the tree as the root. The following example shows a rooted tree with height $3$.
\begin{example}\label{example 1}
\begin{figure}[h]
\centering
\begin{tikzpicture}
[scale=.55]
\draw [fill] (4,6) circle [radius=0.1];
\draw [fill] (3,4) circle [radius=0.1];
\draw [fill] (5,4) circle [radius=0.1];
\draw [fill] (2,2) circle [radius=0.1];
\draw [fill] (4,2) circle [radius=0.1];
\draw [fill] (6,2) circle [radius=0.1];
\draw [fill] (3,0) circle [radius=0.1];
\draw [fill] (4,0) circle [radius=0.1];
\node at (4.45,6.4) {$z$};
\node at (2.3,4) {$y_1$};
\node at (5.6,4) {$y_2$};
\node at (1.5,2) {$y_3$};
\node at (4.5,2) {$y_4$};
\node at (6.5,2) {$y_5$};
\node at (2.4,0) {$y_6$};
\node at (4.6,0) {$y_7$};
\draw (4,6)--(3,4)--(4,2)--(3,0);
\draw (3,4)--(2,2);
\draw (4,2)--(4,0);
\draw (4,6)--(5,4)--(6,2);
\end{tikzpicture}\caption{A rooted tree $T$.}\label{figure 1}
\end{figure}
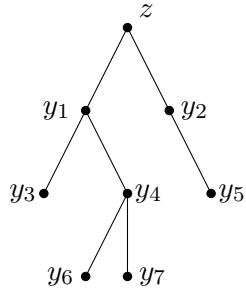
Consider the rooted tree $T$ in \Cref{figure 1} with the vertex $z$ as root. We have $\mathrm{level}(y_i)=1$ for $i=1,2$, $\mathrm{level}(y_i)=2$ for $i=3,4,5$, and $\mathrm{level}(y_i)=3$ for $i=6,7$. Also, $\mathrm{height}(T)=3$. 
\end{example}

Now we are ready to determine the regularity of the closed neighborhood ideal of trees.
    \begin{theorem}\label{tree theorem}
     Let $T$ be a tree. Then,
        \[
        \reg(R_T/NI(T))=a_T.
        \]
    \end{theorem}

    \begin{proof}
    A tree is a connected forest and therefore, by \cite[Theorem 2.5]{SM}, we have $\reg(R_{T}/NI(T)) \geq a_{T}$. Thus it is enough to show that, for any tree $T$, $\reg(R_T/NI(T))\le a_T$, i.e., $a_{T}$ is the upper bound for $\reg(R_{T}/NI(T))$. We prove this assertion by induction on the number of vertices of $T$, i.e., $|V(T)|$. Let $V(T)=\{x_1,\ldots,x_n\}$. If $n\le 2$, then either $NI(T)=\langle 0 \rangle$ or $NI(T)=\langle x_1x_2\rangle$. In this case, note that, if $E(T)\neq \emptyset$, then
    \[
    0\rightarrow R_T(-2)\xrightarrow{\mu}R_T\xrightarrow{\eta}R_T/NI(T)\rightarrow 0
    \]
    is a graded minimal free resolution of $R_T/NI(T)$, where $\mu$ is the multiplication map by $x_1x_2$ and $\eta$ is the natural quotient map. Thus $\reg(R_T/NI(T))=1=a_T$. Hence, $\reg(R_T/NI(T))= a_T$. Therefore, we may assume that $|V(T)|\ge 3$.
    
     First, consider the case when $T$ is a star graph. Then, by \Cref{star}, $\reg(R_T/NI(T))=1=a_T$. Now suppose
     $T$ is not a star graph. From now onwards we consider $T$ to be a rooted tree with some fixed vertex, say $z$, as the root of $T$. Now let $x_1$ be a vertex of $T$ such that $\mathrm{level}(x_1)=\mathrm{height}(T)$. Clearly, $x_1$ is a leaf. Let $N_T(x_1)=\{y\}$. It is easy to see that the neighborhood of $y$ in T is the set $N_T(y)=\{b,x_1,\ldots,x_t\}$, where $b$ is not a leaf and each $x_i$ is a leaf of $T$ for $i\in [t]$. Note that such a non-leaf vertex $b$ exists since we assumed that $T$ is not a star graph. Now the closed neighborhood ideal $NI(T)$ of $T$ can be expressed as follows. 
        \begin{align}\label{eq1}
        NI(T)=\langle\{ x_1y,\ldots,x_ty, m_{(T,b)},m_{(T,w)}: w \in V(T)\setminus  N_T[y] \}\rangle. 
        \end{align}
    Since $y\in N_T(b)$, we have $y\mid m_{(T,b)}$. However, for each $w\notin N_T[y]$, $y$ does not divide the monomial $m_{(T,w)}$. Hence,
        \begin{align}\label{eq2}
        ( NI(T):y)=\left\langle\left\{ x_1,\ldots,x_t, \frac{m_{(T,b)}}{y},m_{(T,w)}: w \in V(T)\setminus  N_T[y] \right\}\right\rangle,
        \end{align} 
        and
       \begin{align*}
       \langle NI(T),y\rangle=\left\langle\{ y,m_{(T,w)}: w \in V(T)\setminus  N_T[y]\}\right \rangle.
       \end{align*}    
       
     Now let us construct a new tree $T'$ by removing the vertices $y, x_{1}, \ldots , x_{t}$ and the corresponding adjacent edges from $T$, i.e., $T'=T\setminus \{y,x_1,\ldots,x_t\}$. Observe that $N_{T'}[b]=N_T[b]\setminus\{y\}$. Hence, the monomial $m_{(T',b)}=\frac{m_{(T,b)}}{y}$. Now if $w$ is a vertex of $T'$ such that $w\neq b$, then $N_{T'}[w]=N_T[w]$. Consequently, $m_{(T',w)}=m_{(T,w)}$. Therefore, the closed neighborhood ideal $NI(T')$ of $T'$ is
    \begin{align}\label{eq4}
    NI(T')=\left\langle \left\{\frac{m_{(T,b)}}{y}, m_{(T,w)} : w \in V(T)\setminus  N_T[y] \right\}\right\rangle.
    \end{align}
By combining \Cref{eq2} and \Cref{eq4}, we get that the colon ideal $(NI(T):y)$ can be expressed as the sum of the ideals $\langle x_1,\ldots,x_t \rangle$ and $NI(T')$. Note that $V(T')\cap \{x_1,\ldots,x_t\}=\emptyset$. Thus by \Cref{reg sum}, $\reg(R_T/(NI(T):y))=\reg(R_{T'}/NI(T'))+\reg(\mathbb K[x_1,\ldots,x_n]/\langle 
x_1,\ldots,x_n \rangle)$. It is easy to see that
\[
0\rightarrow \mathbb K[x_1](-1)\xrightarrow{\mu}\mathbb K[x_1]\xrightarrow{\eta}\mathbb K[x_1]/\langle x_1 \rangle\rightarrow 0
\]
is a minimal free resolution of $\mathbb K[x_1]/\langle x_1\rangle $, where $\mu$ is the multiplication map by $x_1$ and $\eta$ is the natural quotient map. Hence, $\reg(\mathbb K[x_1]/\langle x_1\rangle)=0$. Therefore, by a repeated application of \Cref{reg sum} we have, $\reg(\mathbb K[x_1,\ldots,x_n]/\langle 
x_1,\ldots,x_n \rangle)=0$. Thus $\reg(R_{T'}/NI(T'))=\reg(R_T/(NI(T):y))$.
    
     Now let $T'' =T \setminus \{x_1,\ldots,x_t\}$. Note that $y$ is a leaf of $T''$ with $N_{T''}(y)=\{b\}$. Hence, the monomial $m_{(T'',y)}$ divides $m_{(T'',b)}$. Moreover, for each $w\notin N_T[y]$,  $N_T[w]=N_{T''}[w]$  and hence, $m_{(T,w)}=m_{(T'',w)}$. Therefore, the closed neighborhood ideal of $T''$ is 
    \[
    NI(T'') = \langle\{m_{(T'',y)},m_{(T'',w)}: w \in V(T)\setminus  N_T[y]\}\rangle.
    \]
    Since $x_i\notin V(T'')$ for all $i$, we have $\langle NI(T''),y \rangle=\langle NI(T),y \rangle$. Hence, $\reg(R_T/(\langle NI(T),y \rangle)=\reg(R_{T''}/\langle NI(T''),y \rangle)$. We also note that the variable $y$ appears in the generator $m_{(T'',y)}$ of $NI(T'')$. Therefore, by \Cref{lemma1} we have,
\begin{align}\label{eq5}
\reg(R_T/NI(T))\le\max\{\reg(R_{T'}/NI(T'))+1,\reg(R_{T''}/(\langle NI(T''),y \rangle)\}.
\end{align}

We now use induction on the number of vertices of $T'$ and $T''$ to find the upper bounds on the regularity of their closed neighborhood ideals in terms of their matching numbers. Let $M$ be a matching of $T'$ such that $|M| = a_{T'}$. Observe that as $y,x_{i} \notin V(T')$ for all $i = 1, \dots t$, the set $\{\{y, x_{1}\}\} \cup M$ is a matching of T. This implies $a_T\ge |M|+1= a_{T'} +1$. Since $|V(T')|<|V(T)|$, by induction hypothesis we have, $\reg(R_{T'}/NI(T'))\le a_{T'}<a_T$. Therefore, $\reg(R_{T'}/NI(T'))+1\le a_T$. 

Applying \Cref{lemma1} on the ideal $NI(T'')$ we have, $\reg(R_{T''}/\langle NI(T''),y\rangle)\le \reg(R_{T''}/NI(T''))$. Now let $M''$ be a matching of $T''$ such that $M''=a_{T''}$. Then, $M''$ is also a matching of $T$. Hence, $a_T\ge a_{T''}$. Since $|V(T'')|<|V(T)|$, by induction hypothesis we have, $\reg(R_{T''}/ NI(T''))\le a_{T''}\le a_T$. Hence, $\reg(R_{T''}/\langle NI(T''),y\rangle)\le a_T$. Therefore, from \Cref{eq5} and from the above discussions we obtain $\reg(R_T/NI(T))\le a_T$. This completes the proof of the theorem.
    \end{proof}

        Using \Cref{component lemma} and the above theorem, we now give a proof of \cite[Conjecture 2.11]{SM} as follows.
    \begin{customthm}{\ref{conjecture proof}}
        Let $G$ be a forest. Then, $\reg(R_G/NI(G))=a_G$, where $a_G$ is the matching number of $G$.
    \end{customthm}
    
    \begin{proof}
        Let $G$ be a forest and let $\mathcal C_1,\ldots,\mathcal C_t$ be the connected components of $G$. Then, each $\mathcal C_i$ is a tree. Hence, by \Cref{tree theorem}, $\reg(R_{\mathcal C_i}/NI(\mathcal C_i))=a_{\mathcal C_i}$. Therefore, using \Cref{component lemma} we get, $\reg(R_G/NI(G))=a_G$.
    \end{proof}

    As mentioned above, Sharifan and Moradi in \cite[Theorem 2.5]{SM} showed that for forests, the regularity of the closed neighborhood ideal is bounded below by the matching number. We now show that it can be proved for all graphs using a recent topological result by Matsushita and Wakatsuki \cite{MaWa}. 
    %, let's first review some necessary definitions. 
    Before proving this result, let us recall briefly some facts from the Stanley-Reisner theory of square-free monomial ideals. We refer the readers to \cite{ StanleyBook, RHV} for quick references.
    
    Any square-free monomial ideal $I$ in the polynomial ring $R=\mathbb K[x_1,\ldots,x_n]$ can be realized as the Stanley-Reisner ideal of a simplicial complex $\Delta$, and vice-versa. Recall that, an {\it (abstract) simplicial complex} $\Delta$ on $V=\{x_1,\ldots,x_n\}$ is a collection of subsets of $V$, that satisfies the properties that $\{x_i\} \in \Delta$ for all $i=1,\ldots,n$ and if $F \in \Delta$ and $G \subseteq F$, then $G \in \Delta$. For $W \subseteq V$, we write 
\[
\mathbf x_W = \prod_{x_i \in W} x_i,
\]
to denote the monomial in $R$ obtained by multiplying together all the variables corresponding to the vertices in $W$. Now given a square-free monomial ideal $I$, the complex $\Delta(I)=\{F\subseteq V: \mathbf x_F\notin I\}$ is called the {\it Stanley-Reisner complex} of $I$. Conversely, for a simplicial complex $\Delta$, the ideal $I_{\Delta}=\langle \mathbf x_{F}:F\notin\Delta\rangle$ is called the {\it Stanley-Reisner ideal} of $\Delta$ and it is easy to check that $\Delta(I_{\Delta})=\Delta$ and $I_{\Delta(I)}=I$. This correspondence between square-free monomial ideals and the associated simplicial complexes enables one to study the algebraic properties of $I$ using the topological properties of $\Delta(I)$. In the next theorem, we use this correspondence by first determining the Stanley-Reisner complex of $NI(G)$.

 Let $G$ be a finite simple graph. A subset $W\subseteq V(G)$ is called a {\it dominating set} in $G$ if each $v\in V(G)$ is either contained in $W$ or adjacent to an element in $W$. Following \cite{MaWa}, we define the {\it dominance complex},
\[
\D(G)=\{W\subseteq V(G): W^c\text{ is a dominating set in }G\}.
\]
It is easy to see that $\D(G)$ is an abstract simplicial complex, and the Stanley-Reisner ideal of $\D(G)$ is $NI(G)$ (see \cite[Lemma 2.1]{MaWa}).

\begin{customthm}{\ref{thm:conj}}
    For any graph $G$, $\reg(R_G/NI(G))\ge a_G$, where $a_G$ is the matching number of $G$.
\end{customthm}

\begin{proof}
    By Hochster's formula \cite{Hochster}, 
    \[
        \reg(R_G/I_{\D(G)})\ge (\text{h-}\dim(\D(G)))+1.
    \]
    Here $\text{h-}\dim(\D(G))=\max\{i : \widetilde{H}_i(\D(G);\mathbb Z)\neq 0\}$ is the {\it homological dimension} of $\D(G)$, and $\widetilde{H}_i(\D(G);\mathbb Z)$ denotes the $i^{th}$-reduced homology with coefficient in $\mathbb Z$ of the chain complex associated to $\D(G)$. Furthermore, using \cite[Theorem 2]{lovasz} and \cite[Theorem 1.1 $\&$ Lemma 2.3]{MaWa}, we get that 
    \[
        \chi(G^c) \ge n-(\text{h-}\dim(\D(G)))-1,
    \]
    where $\chi(G^c)$ is the chromatic number of the complement graph $G^c$. Thus, 
    \[
        \reg(R_G/I_{\D(G)})\ge n-\chi(G^c).
    \]
    Observe that $n\ge\chi(G^c)+a_G $. Therefore, 
    \[
        \reg(R_G/I_{\D(G)})\ge a_G.
    \]
    We conclude the proof by observing that $I_{\D(G)}=NI(G)$.
\end{proof}

\begin{remark}\label{complete graph remark}
        Motivated by \Cref{conjecture proof} and \Cref{thm:conj}, one may ask whether $\reg(R_G/NI(G))$ is the same as $a_G$ for any chordal graph $G$. However, this is not true, and as an example we can take a complete graph on $3$ or more vertices. If $K_m$ denotes the complete graph on $m$ vertices $\{x_1,\ldots,x_m\}$, then $NI(K_m)=\langle x_1\cdots x_m  \rangle$. Hence, 
        \[
        0\rightarrow R_{K_m}(-m)\xrightarrow{\mu} R_{K_m}\xrightarrow{\eta} R_{K_m}/NI(K_m)\rightarrow 0
        \]
        is the minimal free resolution of $R_{K_m}/NI(K_m)$, where $\mu$ is the multiplication map by $x_1\cdots x_m$ and $\eta$ is the quotient map. Therefore, $\reg(R_{K_m}/NI(K_m))=m-1$. Note that $a_{K_m}=\lfloor 
 \frac{m}{2}\rfloor$. Thus for $m\ge 3$, $\reg(R_{K_m}/NI(K_m))\neq a_{K_m}$.
    \end{remark}

%{\color{red} Ends here}
    
We now proceed to show that the matching number $a_G$ is also a lower bound for the projective dimension of $R/NI(G)$ for certain classes of graphs. To get our results, we mainly use the concept of Betti splitting which we recall here.

    \begin{definition}\cite{FHVT}
        Let $I,J$ and $K$ be monomial ideals such that $\G(I)=\G(J)\sqcup\G(K)$. Then, $I=J+K$ is called a Betti splitting if 
        \[
        \beta_{i,j}(I)=\beta_{i,j}(J)+\beta_{i,j}(K)+\beta_{i-1,j}(J\cap K),
        \]
        for all non-negative integers $i$ and degrees $j$.
    \end{definition}
    For splitting of monomial ideals Francisco, H\`a, and Van Tuyl \cite{FHVT} proved the following.
    \begin{theorem}\cite[Corollary 2.7]{FHVT}\label{Betti splitting}
        Let $I\subseteq R=\mathbb K[x_1,\ldots,x_n]$ be a monomial ideal. Fix a variable $x_i$ and set
        \[
        J=\langle \{  m\in\G(I): x_i\mid m \} \rangle \text{  and  }K=\langle \{ m\in\G(I): x_i\nmid m \} \rangle.
        \]
        If $J$ has a linear resolution, then $I=J+K$ is a Betti splitting.
    \end{theorem}

    We use the following formula of projective dimension for Betti splitting of monomial ideals.

    \begin{theorem}\cite[cf. Corollary 2.2]{FHVT}\label{regularity and projective dimension}
        Let $I=J+K$ be a Betti splitting of the monomial ideal $I$. Then,
        \[
           \pd(R/I)=\max\{\pd(R/J),\pd(R/K),\pd(R/J\cap K)+1\}.
      \]
    \end{theorem}
    Using the above formula, we obtain the following interesting result concerning the projective dimension of the closed neighborhood ideal of a graph that contains a leaf.

\begin{theorem}\label{chordal closed neighborhood}
Let $G$ be a graph and $x\in V(G)$ be a simplicial vertex of $G$. Then $NI(G)$ admits a Betti splitting. Furthermore, if $x$ is a leaf then $\pd(R_{G'}/NI(G'))\ge a_{G'}$ implies $\pd(R_G/NI(G))\ge a_G$, where $G'=G\setminus N_G[x]$.
    \end{theorem}

    \begin{proof}
Without loss of generality, let $V(G)=\{x_1,\ldots,x_n\}$, where $x_1$ is a simplicial vertex with $N_G(x_1)=\{x_2,\ldots,x_r\}$ for some $r\ge 2$. Then, it is easy to see that $m_{(x_1,G)} \in \G(NI(G))$ and either $m_{(x_i,G)}=m_{(x_1,G)}$ or $m_{(x_i,G)} \notin \G(NI(G))$ for each $i\in[r]$. Now define the ideals
     \[
     J=\langle m_{(G,x_1)}\rangle \text{  and  }K=\langle\{ m_{(G,x_i)}: i>r\}\rangle.
     \]
     Then, $J=\langle \{ m\in\G(NI(G)): x_1\mid m \} \rangle$ and $K=\langle \{ m\in\G(NI(G)): x_1\nmid m \} \rangle$. Moreover, $NI(G)=J+K$ with $\G(NI(G))=\G(J)\sqcup \G(K)$. Therefore, by \Cref{Betti splitting}, we have that $NI(G)=J+K$ is a Betti splitting. Hence, using \Cref{regularity and projective dimension} we obtain
     \[
         \pd(R_G/NI(G))=\max\{\pd(R_G/J),\pd(R_G/K),\pd(R_G/J\cap K)+1\}.
     \]
     In particular, $\pd(R_G/NI(G))\ge \pd(R_G/J\cap K)+1$. 
     
     We proceed to show that $J\cap K=\prod_{i=1}^rx_i\cdot NI(G')$, where $G'=G\setminus N_G[x_1]$. Let  $m_{(G,x_t)}\in K$, which implies that $t>r$. For such an $x_t$, if $x_t\in \cup_{i=1}^rN_G[x_i]$, then $N_G[x_t]=N_{G'}[x_t]\sqcup (N_G[x_t]\cap N_G(x_1))$, and if $x_t\notin \cup_{i=1}^r N_G[x_i]$, then $N_G[x_t]=N_{G'}[x_t]$. Thus in both the cases, $\mathrm{lcm}(m_{(G,x_1)},m_{(G,x_t)})=\prod_{i=1}^rx_i\cdot m_{(G',x_t)}$. Since $J\cap K$ is generated by $\mathrm{lcm}(m_{(G,x_1)},m_{(G,x_t)})$ for $t>r$, we see that $J\cap K=\prod_{i=1}^rx_i\cdot NI(G')$.
          
          For the second part of the theorem, we consider $x_1$ to be a leaf of $G$. In that case, $N_{G}(x_1)=\{x_2\}$ and $G'=G\setminus\{x_1,x_2\}$. Thus $J\cap K=x_1x_2\cdot NI(G')$. Hence, by applying the \Cref{extra variable} twice we obtain that $\pd(R_G/J\cap K)=\pd(R_{G'}/NI(G'))$. Thus, $\pd(R_G/NI(G))\ge\pd(R_{G'}/NI(G'))+1$. 

         Now let $M\subseteq E(G)$ be such that $|M|=a_G$, where $G$ contains the leaf vertex $x_1$. Then, $|M\cap E(G'')|\le 1$, where $G''$ is a subgraph of $G$ on the vertex set $V(G'')=V(G)$ with edge set $E(G'')=\{e\in E(G): e\cap \{x_1,x_2\}\neq\emptyset\}$. Also, $M\cap E(G')$ is a matching of $G'$. Note that $M=(M\cap E(G')\sqcup(M\cap E(G'')) $ and thus $a_G=|M|\le a_{G'}+1$. Since $\pd(R_G/NI(G))\ge\pd(R_{G'}/NI(G'))+1$ and $\pd(R_{G'}/NI(G'))\ge a_{G'}$, we have $\pd(R_G/NI(G))\ge a_G$. This completes the proof.
    \end{proof}

    In \cite[Theorem 2.5]{SM} Sharifan and Moradi showed that if $G$ is a forest, then $\pd(R_G/NI(G))\ge a_G$. As an application of \Cref{chordal closed neighborhood}, we give an alternate proof of this fact below.

    \begin{corollary}\label{forest pd}
        Let $F$ be a forest. Then, $\pd(R_F/NI(F))\ge a_F$.
    \end{corollary}
    \begin{proof}
        Let $V(F)=\{x_1,\ldots,x_n\}$. The proof is by induction on $n$. If $n\le 2$, then $NI(F)=\l x_1x_2\r$ or the zero ideal. In both cases $\pd(R_F/NI(F))\ge a_F$. Therefore, we may assume that $n\ge 3$. If $F$ consists of only isolated vertices, then clearly, $\pd(R_F/NI(F))=0=a_F$. Otherwise, $F$ contains a leaf vertex, say $x$. Note that $F'=F\setminus N_F[x]$ is also a forest and hence, by induction hypothesis, $\pd(R_{F'}/NI(F'))\ge a_{F'}$. Consequently, by \Cref{chordal closed neighborhood} we have, $\pd(R_F/NI(F))\ge a_F$.
    \end{proof}
    \begin{remark}\label{rem:pdoftree}
        Unlike \Cref{conjecture proof}, we do not have, in general, $\pd(R/NI(T))= a_T$, even if $T$ is a tree. For example, let $T$ be the path graph on the vertex set $\{x_1,x_2,x_3\}$. Then, $NI(T)=\l x_1x_2,x_2x_3\r$. Therefore, $\pd(R/NI(T))=2>1=a_T$.       
    \end{remark}

    In the next theorem, we show that $a_G$ provides a lower bound for the projective dimension of $R_G/NI(G)$, when $G$ is a unicyclic graph. Note that if $G$ is just a cycle $C_n$ of length $n$, then $NI(C_n)$ is nothing but the well-known $3$-path ideal of $C_n$. In general, let $G$ be a finite simple graph on the vertex set $\{x_1,\ldots,x_n\}$ and let $t$ be a positive integer. A path of length $t$ from a vertex $u$ to another vertex $v$ is a sequence of vertices $u=x_{i_1},x_{i_2},\ldots,x_{i_t}=v$ such that $\{x_{i_j},x_{i_{j+1}}\}\in E(G)$ for each $j\in [t-1]$. The {\it $t$-path ideal} of $G$, denoted by $J_t(G)$, is a monomial ideal generated by the monomials $\left\{\prod_{j=1}^tx_{i_j}: \{x_{i_1},x_{i_2},\ldots,x_{i_t}\} \text{ is a path in }G \right\}$. Now if $G=C_n$, then it is easy to see that $NI(C_n)=J_3(C_n)$. Alilooee and Faridi \cite{AlFa} determined the Betti numbers and the projective dimension of the $t$-path ideals of cycles. Using their formula in the case $t=3$ and using \Cref{chordal closed neighborhood}, we can show that the projective dimension of the closed neighborhood ideal of a unicyclic graph is bounded below by the matching number of the graph.

    \begin{theorem}\label{unicyclic}
        Let $G$ be a unicyclic graph. Then, $\pd(R_G/NI(G))\ge a_G$.
    \end{theorem}
    \begin{proof}
        We prove this by induction on $|V(G)|$. If $|V(G)|\le 2$, then we see that $\pd(R_G/NI(G))\ge a_G$. Therefore, we may assume that $|V(G)|\ge 3$.

         First consider the case when $G=C_n$, a cycle of length $n$.  Note that the matching number $a_{C_n}=\lfloor \frac{n}{2} \rfloor$. Now since $NI(C_n)=J_3(C_n)$, by \cite[Corollary 5.5]{AlFa} we get $\pd(R_{C_n}/NI(C_n))=\frac{n}{2}$ if $n$ is divisible by $4$ and $\pd(R_{C_n}/NI(C_n))=\frac{n-d+2}{2}$, otherwise, where $d$ is the remainder when $n$ is divided by $4$. Thus in both cases $\pd(R_{C_n}/NI(C_n))\ge a_{C_n}$

        Now suppose $G$ is not a cycle. Then, $G$ contains a leaf say $x$ with $N_G(x)=\{y\}$. Note that $G'=G\setminus N_G[x]$ is either a forest or a unicyclic graph. If $G'$ is a forest, then by \Cref{forest pd} we have, $\pd(R_{G'}/NI(G'))\ge a_{G'}$. Now suppose $G'$ is a unicyclic graph. Since $|V(G')|<|V(G)|$, by the induction hypothesis we have, $\pd(R_{G'}/NI(G'))\ge a_{G'}$. Hence, using \Cref{chordal closed neighborhood} we obtain, $\pd(R_G/NI(G))\ge a_G$. This completes the proof.
    \end{proof}

\begin{remark}\label{cycle remark}
    If $G=C_5$, the cycle graph of length $5$, then $\pd(R_{C_5}/NI(C_5))=3>a_{C_5}=2$. Also, if $G=C_7$, the cycle graph of length $7$, then $\reg(R_{C_7}/NI(C_7))=4>a_{C_5}=3$. Thus both the inequalities in \Cref{unicyclic} and \Cref{thm:conj} could be strict in the case of unicyclic graphs.
\end{remark}
\begin{remark}\label{chordal not bounded below}
    We do not have, in general, $\pd(R_G/NI(G))\ge a_G$, even if $G$ is a chordal graph. For example, let $G=K_m$, the complete graph on $m$ vertices, where $m\ge 4$. Then, from  \Cref{complete graph remark} we have $\pd(R_{K_m}/NI(K_m))=1$. However, $a_{K_m}=\lfloor \frac{m}{2}\rfloor$. Hence, $a_{K_m}>\pd(R_{K_m}/NI(K_m))$.
\end{remark}

We end this section by comparing the projective dimension and the matching number of the wheel graphs. Let $W_{n+1}$ denote the wheel graph on $n+1$ vertices $\{x,y_1,\ldots,y_n\}$ with the edge set $E(W_{n+1})=\{\{x,y_i\},\{y_j,y_{j+1}\},\{y_1,y_n\}: 1\le i\le n,1\le j\le n-1\}$. It is easy to see $a_{W_{n+1}}=\lfloor\frac{n+1}{2}\rfloor$. Moreover, $NI(W_{n+1})=x\cdot NI(C_n)$, where $C_n$ is the cycle $y_1\cdots y_n$.
Hence, by the proof of \Cref{unicyclic} and by \Cref{extra variable} we obtain, $\pd(R_{W_{n+1}}/NI(W_{n+1}))=\frac{n}{2}$ if $n$ is divisible by $4$ and $\pd(R_{W_{n+1}}/NI(W_{n+1}))=\frac{n-d+2}{2}$, otherwise, where $d$ is the remainder when $n$ is divided by $4$. Thus comparing the formulas above, we obtain the following.

\begin{corollary}\label{wheel result}
    Let $W_{n+1}$ denote the wheel graph on $n+1$ vertices. Then, 
    \begin{equation*}
        \pd(R_{W_{n+1}}/NI(W_{n+1}))=
        \begin{cases}
            a_G & \text{if } n\not\equiv 3\pmod 4,\\
            a_G-1 & \text{if } n\equiv 3\pmod 4.
        \end{cases}
    \end{equation*}
    %$\pd(R_{W_{n+1}}/NI(W_{n+1}))= a_G$ if $n\not\equiv 3\pmod 4$ and 
\end{corollary}

\section{Concluding remarks}\label{section 4}
In \Cref{thm:conj}, we saw that the topological results \cite[Theorem 1.1 $\&$ Lemma 2.3]{MaWa} of Matsushita and Wakatsuki can be used to show that the inequality $\reg(R_G/NI(G))\ge a_G$ is true for all graphs. Moreover, the idea of Betti splitting can also be used to show the same result related to regularity for chordal graphs and unicyclic graphs. The proofs would be similar to that of projective dimension case for trees (by replacing leaf to a simplicial vertex) and unicyclic graphs. Thus, the following is a natural question to ask. 

\begin{question}
    Is there an algebraic proof of \Cref{thm:conj}?
\end{question}

The forests and the unicyclic graphs are the two classes of graphs for which we proved that $\pd(R_G/NI(G))\ge a_G$. Interestingly, for chordal graphs, we do not have this inequality, and the simplest example is any complete graph on at least $4$ vertices. Given this, we ask the following.

\begin{question}
    For which classes of graphs the projective dimension of the quotient of the closed neighborhood ideal is bounded below by the matching number of the corresponding graph?
\end{question}

As mentioned above, in \Cref{thm:conj} we have shown that if $G$ is any finite simple graph then $\reg(R_G/NI(G))\ge a_G$. Moreover, for forests, we have proved in \Cref{conjecture proof} that $\reg(R_G/NI(G))=a_G$. However, for various other classes of graphs, the equality does not hold as seen in \Cref{section 3}. It is worth mentioning that even in the case of complete graphs, the difference between the regularity of the closed neighborhood ideal and the matching number can be made arbitrarily large (see \Cref{complete graph remark}). In the case of complete bipartite graphs, the following example shows that the equality between regularity and the matching number does not hold, and the difference can also be made arbitrarily large. 

\begin{example}
    Let $G=K_{n,2}$ with the vertex set $\{x_1,\ldots,x_n,y_1,y_2\}$. Then, $a_G=2$. However, the ideal $\langle NI(G),y_2\rangle=\langle y_2, x_1x_2\cdots x_ny_1\rangle$. Note that $\reg(R_G/\langle y_2, x_1x_2\cdots x_ny_1\rangle)=n$. Hence, by \Cref{lemma1}, $\reg(R_G/NI(G)\ge n$.
\end{example}

Given the above observations, the following is a natural question to ask.

% \begin{question}
%     Classify all graphs $G$ such that the difference between $\reg(R_G/NI(G))$ and $a_G$ is arbitrarily large.
% \end{question}

% The following is another similar question that would be worth exploring.

\begin{question}
    Classify all graphs $G$ such that $\reg(R_G/NI(G))=a_G$.
\end{question}

\section*{Acknowledgements}
The authors would like to thank the referees for their careful reading and helpful suggestions. The authors would also like to express their gratitude to Yusuf Civan for bringing reference \cite{MaWa} to their attention, as well as for pointing out the connection of \Cref{thm:conj} with this reference. Shiny Chakraborty would like to thank SERB, DST for support through scholarship during her stay at IIT Bhilai. Ajay P. Joseph would like to express his gratitude to the National Institute of Technology Karnataka for the Doctoral Research Fellowship. Amit Roy is partially supported by a grant from the Infosys Foundation and a postdoctoral research fellowship from CMI, India. He also acknowledges financial support from DAE, Government of India for a postdoctoral research fellowship during his stay in NISER Bhubaneswar, India.  Work on this project started when AR visited AS at the Indian Institute of Technology Bhilai, India. They thank IIT Bhilai for the hospitality. Anurag Singh is partially supported by the Start-up Research Grant SRG/2022/000314 from SERB, DST, India.

\subsection*{Data availability statement} Data sharing is not applicable to this article as no new data were created or
analyzed in this study.

\subsection*{Conflict of interest} The authors declare that they have no known competing financial interests or personal
relationships that could have appeared to influence the work reported in this paper.

\bibliographystyle{abbrv}

\end{document}